\documentclass[11pt]{amsart}
\usepackage{amsmath,amssymb}
\usepackage{color}
\usepackage[all,curve,frame]{xy}

\usepackage[pdftex]{graphicx}

\newtheorem{defi}{Definition}[section]
\newtheorem{sinobservacion}[defi]{}
\newenvironment{sinob}{\begin{sinobservacion} \rm}{\end{sinobservacion} }
\newtheorem{coro}[defi]{Corollary}
\newtheorem{lem}[defi]{Lemma}
\newtheorem{obs}[defi]{Remark}
\newtheorem{nota}[defi]{Notation}
\newtheorem{prop}[defi]{Proposition}
\newtheorem{teo}[defi]{Theorem}

\newcommand{\benu}{\begin{enumerate}}
\newcommand{\enu}{\end{enumerate}}
\newcommand{\fle}{\rightarrow}

\newcommand{\C}{\mathcal{C}}
\newcommand{\D}{\mathcal{D}}
\newcommand{\til}{\widetilde}
\newcommand{\phii}{\varphi}
\newcommand{\mbmA}{\mbox{mod}\,A}
\newcommand{\emmA}{\emph{mod}\,A}
\newcommand{\mbmH}{\mbox{mod}\,H}

\newcommand{\mbmG}{\mbox{mod}\,\Gamma}
\newcommand{\emmG}{\emph{mod}\,\Gamma}
\newcommand{\mbiG}{\mbox{ind}\,\Gamma}
\newcommand{\emiG}{\emph{ind}\,\Gamma}
\newcommand{\mbiA}{\mbox{ind}\,A}

\newcommand{\mbeCT}{\mbox{End}_\mathcal{C}(T)^{op}}
\newcommand{\emeCT}{\emph{End}_\mathcal{C}(T)^{op}}

\newcommand{\al}{\alpha}
\newcommand{\Gaa}{\Gamma_A}

\begin{document}

\title[On the radical of Cluster tilted algebras]{On the radical of Cluster tilted algebras}
\author[Chaio Claudia]{Claudia Chaio}
\address{Centro Marplatense de Investigaciones Matem\'aticas, Facultad de Ciencias Exactas y
Naturales, Funes 3350, Universidad Nacional de Mar del Plata y CONICET 7600 Mar del
Plata, Argentina}
\email{claudia.chaio@gmail.com}

\author[Guazzelli Victoria]{Victoria Guazzelli }
\address{Centro Marplatense de Investigaciones Matem\'aticas, Facultad de Ciencias Exactas y
Naturales, Funes 3350, Universidad Nacional de Mar del Plata, 7600 Mar del
Plata, Argentina}
\email{victoria.guazzelli@hotmail.com}
\keywords{irreducible morphisms, radical, projective cover, injective hull.}
\subjclass[2000]{16G70, 16G20, 16E10}

\maketitle

\begin{abstract} We determine the minimal lower bound $n$, with  $n \geq 1$, where the $n$-th power of the radical of the module category of a representation-finite
cluster tilted algebra  vanishes. We give such a bound in terms of
the number of  vertices of the underline quiver. Consequently, we get the nilpotency index  of the radical of the module category
for representation-finite self-injective cluster tilted algebras.
We also study the non-zero composition of $m$, $m \ge 2$, irreducible morphisms
between indecomposable modules in representation-finite cluster tilted algebras lying in the $(m+1)$-th power of the radical of their module
category.
\end{abstract}

\section*{Introduction}

Let $A$ be a finite dimensional algebra over an algebraically closed field.
The representation theory of an algebra $A$
deals with the study of the module category of
finitely generated $A$-modules, $\mbox{mod}\,A$. A fundamental tool in
the study of $\mbox{mod}\,A$ is the Auslander-Reiten theory, based
on irreducible morphisms and almost split sequences.

For $X,Y \in \mbox{mod}\;A$,
we denote by $\Re(X,Y)$ the set of
all morphisms $f: X \rightarrow Y$ such that, for all
indecomposable $A$-module $M$, each pair of morphisms $h:M
\rightarrow X$ and $h':Y\rightarrow M$ the composition $h'fh$ is not an isomorphism. Inductively, the powers of
$\Re(X,Y)$ are defined.

There is a close connection between irreducible morphisms and the
powers of the radical, given by a well-known  result proved by R. Bautista which states that if $f:X \rightarrow Y$ is a morphism between indecomposable $A$-modules then  $f$ is irreducible if and only if $f\in \Re (X, Y) \backslash \Re^{2}(X, Y)$,  see \cite{B}.

In case that $\Re^{n}(M,N)=0$
for some positive integer $n$ and for all $M$ and $N$ in $\mbox{mod}\;A$, we write this fact by the expression $\Re^{n}(\mbox{mod}\;A)=0$.
We recall that an algebra $A$ is representation-finite (or of finite representation type) if and
only if there is a positive integer ${n}$ such that  $\Re^{n}(\mbox{mod}\;A)$ vanishes,  (see \cite{ARS} p. 183).

In \cite{L}, S. Liu defined the notion of degree of an irreducible morphism (see \ref{nilpo}) which has been a powerful tool to study, between others problems, the one concerning the nilpotency of the radical of a module category of an algebra $A$, in case that we deal with finite-dimensional $k$-algebras over an algebraically closed field of finite representation type.

If $A$ is a finite dimensional basic algebra over an algebraically
closed field then we know that  $A\simeq kQ_A/I_A$. In addition,  if $A$ is representation-finite
then by \cite{CLT} we know  that all
irreducible epimorphisms and all irreducible monomorphisms are of
finite left and right degree, respectively. In particular, the
irreducible monomorphism $\iota_a: \mbox{rad} (P_a)\hookrightarrow
P_a$ where $P_a$ is the projective module corresponding to the
vertex $a$ in $Q_A$, has finite right degree. Dually, the
irreducible epimorphism $g_a:I_a \rightarrow
I_a/\mbox{soc}(I_a)$ where $I_a$ is the injective module
corresponding to the vertex $a$ in $Q_A,$ has finite left degree.
We denote by $S_a$ the simple $A$-module corresponding to the
vertex $a$ in $Q_A.$

By \cite{C}, we  know that for a finite
dimensional algebra over an algebraically closed field  $A\simeq kQ_A/I_A$ where $A$ is representation-finite
we can compute the nilpotency index $r_A$
of $\Re(\mbox{mod}\; A)$ by
$\mbox{max}\{r_a\}_{a\in (Q_A)_{0}}+1$ where
$r_a$ is equal to the length of any path of irreducible
morphisms between indecomposable modules from the projective $P_a$ to the injective
$I_a$, going through the simple $S_a$.
\vspace{.05in}

Applying the above mentioned result we give the minimal positive integer $m$ such
that $\Re(\mbox{mod}\;\Gamma)$ vanishes, where $\Gamma$ is a cluster tilted algebra
of type $\overline{\Delta}$, with ${\Delta}$ a Dynkin quiver. More precisely,  we prove Theorem A and B.
\vspace{.1in}

{\bf Theorem A.} {\it Let $\C$ be the cluster category of a representation-finite hereditary algebra $H$.
let $\overline{T}$ be an almost complete tilting object in $\C$ with complements $M$ and $M^*$.
Consider $\Gamma=\emph{End}_{\C}({T})^{op}$ and
$\Gamma'=\emph{End}_{\C}({T'})^{op}$
the cluster tilted algebras with  $T=\overline{T}\oplus M$ and
$T'=\overline{T}\oplus M^*$.  Let $r_{\Gamma}$ and $r_{\Gamma'}$
be the nilpotency indices of $\Re(\emmG)$ and $\Re(\emmG')$, respectively.
Then, $r_{\Gamma}=r_{\Gamma'}$.}
\vspace{.1in}

{\bf Theorem B.} {\it
Let ${\Delta}$ be a Dynkin quiver and let $\Gamma$ be a cluster tilted algebra
of type $\overline{\Delta}$. Let $r_{\Gamma}$ be  the nilpotency index of $\Re(\emmG)$. Then the  following conditions hold.
\begin{enumerate}
\item[(a)] If $\overline{\Delta}=A_n$, then $r_{\Gamma}=n$ for $n\geq 1$.
\item[(b)] If $\overline{\Delta}=D_n$, then $r_{\Gamma}=2n-3$ for $n\geq 4$.
\item[(c)] If $\overline{\Delta}=E_6$, then $r_{\Gamma}=11$.
\item[(d)] If $\overline{\Delta}=E_7$, then $r_{\Gamma}=17.$
\item[(e)] If $\overline{\Delta}=E_8$, then $r_{\Gamma}=29.$
\end{enumerate}}
\vspace{.1in}

We observe that the nilpotency index of the radical of the module category of a cluster tilted algebra
of type $\overline{\Delta}$ with ${\Delta}$ a Dynkin quiver, coincide with the nilpotency index of the radical of the module category of the hereditary algebra $k\Delta$.
\vspace{.05in}

The non-zero composition of $n$
irreducible morphisms between indecomposable modules could belong to $\Re^{n+1}$. In the last years, there have been many works done in this direction. The first to give a partial solution to that problem were K. Igusa and G. Todorov in \cite{I-T},  where they proved that if
$X_0\xrightarrow{f_1}X_1\to\cdots\to X_{n-1}\xrightarrow{f_n}X_n$
is a sectional path then $f_n \dots f_1$ lies in
$\Re^n(X_0,X_n)$ but not in $\Re^{n+1}(X_0,X_n)$.

In \cite{CCT}, F. U. Coelho, S. Trepode and the first named author characterized when the
composition of two irreducible morphisms is  non-zero  and lies
in $\Re^3(\mbox{mod}\;A)$ for $A$ an artin algebra. In
\cite{CLT2}, P. Le Meur, S. Trepode and the first named author solved the problem of when the composition of $n$ irreducible morphisms between indecomposable modules is non-zero and belongs to $\Re^{n+1}(\mbox{mod}\;A)$ for finite dimensional $k$-algebras over a perfect field $k$.

As a  consequence of the results of this work, we obtain  when the composition of $n$ irreducible morphisms between indecomposable $A$-modules belongs to the $n+1$ power of the radical of their module category, for a representation-finite cluster tilted algebra $A$.  More precisely, we prove  Theorem C.
\vspace{.1in}

{\bf Theorem C.} {\it Let $\Gamma$ be a representation-finite cluster tilted algebra. Consider the  irreducible morphisms
$h_i:X_i\fle X_{i+1}$, with $X_i\in \emiG$ for $1\leq i\leq m$.
Then $h_{m}\dots h_{1}\in \Re^{m+1}(X_{1}, X_{m+1})$ if and only if $ h_{m}\dots h_{1} = 0$.}
\vspace{.1in}

\thanks{The authors thankfully acknowledge partial support from
CONICET and Universidad Nacional de Mar del Plata, Argentina.
The authors also thanks Ana Garcia Elsener for useful conversations.
The first author is a CONICET researcher.}

\section{Preliminaries}

Throughout this work, by an algebra we mean a finite dimensional basic $k$-algebra over an algebraically closed field,  $k$.

\vspace{.1in}
\begin{sinob} {\bf Notions on quivers and algebras}
\vspace{.1in}

A {\it quiver} $Q$ is given by a set of vertices $Q_0$ and
a set of arrows $Q_1$, together with two
maps $s,e:Q_1\fle Q_0$. Given an arrow $\al\in Q_1$, we write $s(\al)$
the starting vertex of $\al$ and $e(\al)$
the ending vertex of $\al$. We denote by $\overline{Q}$ the underlying graph of $Q$.
For each algebra $A$ there is a quiver $Q_A$, called the {\it ordinary quiver of}  $A$, such that
$A$ is the quotient of the path algebra $kQ_A$ by an admissible ideal.

Let $A$ be an algebra. We denote by $\mbox{mod}\,A$ the category of finitely generated
left $A$-modules and  by $\mbiA$ the full subcategory of $\mbmA$ which consists of
one representative of each isomorphism class of indecomposable $A$-modules.

We say that $A$ is a \textit{representation-finite algebra} if there is only a finite number
of isomorphisms classes of indecomposable A-modules.

We denote by $\Gaa$ the Auslander-Reiten
quiver of $\mbmA$, and  by $\tau$ the Auslander-Reiten translation DTr with inverse TrD denoted by $\tau^{-1}$.
\end{sinob}
\vspace{.1in}

\begin{sinob}{\bf On the radical of a module category}
\vspace{.1in}

A morphism $f :X  \rightarrow  Y$, with $X,Y \in \mbox{ mod}\,A$,
is called {\it irreducible} provided it does not split
and whenever $f = gh$, then either $h$ is a split monomorphism or $g$ is a
split epimorphism.

If $X,Y \in \mbox{mod}\,A$, the ideal $\Re(X,Y)$ of $\mbox{Hom}(X,Y)$ is
the set of all the morphisms $f: X \rightarrow Y$ such that, for each
$M \in \mbox{ind}\,A$, each $h:M \rightarrow X$ and each $h^{\prime }:Y\rightarrow M$
the composition $h^{\prime }fh$ is not an isomorphism. For $n \geq 2$, the powers of $\Re(X,Y)$
are inductively defined. By $\Re^\infty(X,Y)$ we denote the intersection
of all powers $\Re^i(X,Y)$ of $\Re(X,Y)$, with $i \geq 1$.

By \cite{B}, it is known that for $X,Y \in \mbox{ind}\,A$, a morphism $f :X  \rightarrow  Y$ 
is irreducible if and only if $f \in \Re(X,Y)\setminus \Re^2(X,Y)$.

 We recall the next  proposition fundamental for our results.
\begin{prop}\cite[V, Proposition 7.4]{ARS}\label{aus}
Let $M$ and $N$ be indecomposable modules in $\emmA$ and let $f$ be a
morphism in $\Re^n(M,N)$, with $n\geq 2$. Then, the following conditions hold.
\benu
          \item[(i)]  There exist a natural number  $s$, indecomposables $A$-modules
                   $X_1,\dots,X_s$, morphisms $f_i\in \Re(M,X_i)$ and morphisms
                   $g_i:X_i\fle N$, with each $g_i$  a sum of compositions of $n-1$ irreducible
                    morphisms between indecomposable modules such that $f=\sum_{i=1}^sg_if_i$.
          \item[(ii)]  If  $f\in \Re^n(M,N)\setminus \Re^{n+1}(M,N)$, then at leats one of the $f_i$
          in $(i)$ is irreducible.
          \enu
\end{prop}

It is well known by a result of M. Auslander that an algebra $A$ is representation-finite
if and only if $\Re^\infty(\mbmA)=0$. That is,
there is a positive integer $n$ such that $\Re^n(X,Y)=0$ for all $X,Y$ $A$-modules.
The minimal positive integer  $m$ such that $\Re^m(\mbmA)=0$ is called the \textit{nilpotency index} of
$\Re(\mbmA).$ We denote such an index by $r_A$.
\end{sinob}
\vspace{.1in}

\begin{sinob}{\bf Basic definitions of paths}
\vspace{.1in}

A path  in $\mbmA$ is a sequence $M_0\stackrel{f_{1}}\fle M_1\stackrel{f_{2}}\fle M_2\fle
\dots \fle M_{n-1}\stackrel{f_{n}}\fle M_n$
of non-zero non-isomorphisms $f_1,\dots, f_n$ between indecomposable $A$-modules with $n\geq 1$.
In case that $f_1,\dots,f_n$ are irreducible morphisms, we say that the path
is in $\Gamma_A$ or equivalently that it is a \textit{path in $\Gamma_A$}.
The length of a path in $\Gamma_A$ is defined as the number of irreducible morphisms
(not necessarily different) involved in the path.

Let us recall that paths in $\Gamma_A$ having the same starting vertex and the same ending vertex
are called \textit{parallel paths}.

Let $\Gamma$ be a component of $\Gamma_A$. We say that $\Gamma$ is a \textit{component with length}
if parallel paths in $\Gamma$ have the same length. Otherwise, it is called a \textit{component without length}, see \cite{CPT}.

By a \textit{directed component} we mean a component $\Gamma$ that there is no sequence $M_0\stackrel{f_{1}}\fle M_1\stackrel{f_{2}}\fle M_2\fle
\dots \fle M_{n-1}\stackrel{f_{n}}\fle M_n$
of non-zero non-isomorphisms $f_1,\dots, f_n$ between indecomposable $A$-modules with $M_0 = M_n$.

Given a directed component $\Gamma$ of $\Gamma_A$, its \textit{orbit graph} $O(\Gamma)$ is a graph  defined as follows: the points of $O(\Gamma)$ are the $\tau$-orbits
$O(M)$ of the indecomposable modules M in $\Gamma$. There is an edge between $O(M)$ and $O(N)$ in $O(\Gamma)$
if there are positive integer $n, m$ and  either an irreducible morphism from $\tau^m M$ to  $\tau^n N$ or from $\tau^n N$ to $\tau^m M$ in $\mbox{mod} A$.

Note that if the orbit graph  $O(\Gamma)$ is of tree-type,
then $\Gamma$ is a simply connected translation quiver, and by \cite{BG} we know that
$\Gamma$ is a component with length.
\end{sinob}
\vspace{.1in}

\begin{sinob} \label{nilpo} {\bf On the nilpotency index of the radical of a module category}
\vspace{.1in}

We say that the \textit{depth of a morphism} $f: M \rightarrow N$
in ${\rm mod}\,A$ is infinite if $f\in \Re^\infty(M, N)$; otherwise, the depth
of $f$ is the integer $n\geq 0$ for which $f\in \Re^n(M, N)$ but $f \notin \Re^{n+1}(M, N)$.
We denote the depth of $f$ by ${\rm dp}(f)$.

Next, we recall the definition of degree of an irreducible morphism given by S. Liu in \cite{L}.

Let $f:X\rightarrow Y$ be an irreducible morphism in
$\mbox{mod}\,A$, with $X$ or $Y$ indecomposable. The {\it left degree} $d_l(f)$ of $f$ is infinite,
if for each integer $n\geq 1 $, each module $Z\in \mbox{ind}\,A$
and each morphism $g:Z  \rightarrow X$ with ${\rm dp}(g)=n$ we
have that $fg \notin \Re^{n+2}(Z,Y)$. Otherwise, the left degree of
$f$ is the least natural $m$ such that there is an $A$-module $Z$
and a morphism $g:Z  \rightarrow X$ with ${\rm dp}(g)=m$ such that
$ fg \in \Re^{m+2}(Z,Y)$.

The {\it right degree} $d_r(f)$ of an irreducible morphism $f$
is dually defined.\vspace{.05in}

In order to compute the nilpotency index of the radical of any module category we shall strongly used \cite[Theorem A]{C}.
For the convenience of the reader,  we state below such a result. 

Let $A=kQ_A/I_A$ be a representation-finite algebra. Let $a\in (Q_A)_0$ and  $P_a$, $I_a$ and $S_a$ be the projective, injective and simple $A$-modules, respectively,  corresponding to the vertex $a$.

For each $a\in (Q_A)_0$,
let $n_a$ be the number defined as follows: \benu
\item[$\bullet$] If $P_a=S_a$, then $n_a=0$.
\item[$\bullet$] If $P_a\not \simeq S_a$, then $n_a=d_r(\iota_a)$, where $\iota_a$
is the irreducible morphism $\iota_a:\mbox{rad}(P_a)\fle P_a$.
\enu

Dually, for each $a\in (Q_A)_0$, let $m_a$  be the number defined  as follows:
\benu
\item[$\bullet$] If $I_a=S_a$, then $m_a=0$.
\item[$\bullet$] If $I_a \not \simeq S_a$, then $m_a=d_r(\theta_a)$, where $\theta_a$
is the irreducible morphism $\theta_a:I_a\fle I_a/\mbox{soc}(I_a)$.
\enu

We write $r_a=m_a+n_a$.

\begin{teo} \cite[Theorem A]{C} Let $A\simeq kQ_A/I_A$ be a finite
dimensional algebra over an algebraically closed field and assume
that $A$ is a representation-finite algebra. Then the nilpotency index $r_A$
of $\Re(\emmA)$ is
$r_A=\emph{max}\{r_a\}_{a\in (Q_A)_{0}}+1.$
\end{teo}

Following \cite[Remark 1]{C}, $r_a$ is equal to the length of any path of irreducible
morphisms between indecomposable modules from the projective $P_a$ to the injective
$I_a$, going through the simple $S_a$.
\vspace{.05in}

Finally, we recall a result of \cite{C} that shall be useful in this work.

\begin{lem}\label{clau2.5}\emph{\cite[Lemma 2.4]{C}}
Let $A\cong kQ/I$
 be a representation-finite algebra.
 Given $a\in Q_{0}$, consider $r_a$ the number defined as above. Then, the following conditions hold.
\begin{enumerate}
\item[(a)]
Every non-zero morphism $f:P_{a}\rightarrow I_{a}$ that factors through the simple A-module $S_a$ is
such that ${\rm dp}(f)=r_a$.
\item[(b)] Every non-zero morphism $f:P_{a}\rightarrow I_{a}$ which does not factor through the simple A-
module $S_a$ is such that ${\rm dp}(f)=k$, with $0\leq k<r_a$.
\end{enumerate}
\end{lem}
\end{sinob}
\vspace{.1in}

\begin{sinob} {\bf The cluster category}
\vspace{.1in}

Let $H$ be a hereditary algebra.  We denote by $\D=\mathcal{D}^{b} (\mbox{mod}\,H)$ the
bounded  derived category of $\mbmH$. The \textit{cluster category}, $\C$,  is defined as the quotient  $\mathcal{D}/F $,
where $F$ is the composition $\tau_{\D}^{-1}[1]$ of the suspension functor and the Auslander-Reiten
translation in $\D$. The objects of $\C$ are the $F$-orbits of the objects in $\D$, and the morphisms
in $\C$ are defined as

\begin{equation}\label{morf}
\mbox{Hom}_{\C}(\til{X},\til{Y})=\coprod_{i\in \mathbb{Z}}\mbox{Hom}_{\D}(F^i X, Y),
\end{equation}

\noindent where $X, Y$ are objects in $\D$ and  $\til{X}, \til{Y}$ are the corresponding objects in $\C$.
By \cite[Proposition 1.5]{BMRRT}, the summands of $(\ref{morf})$ are almost all zero.

We recall some basic and useful properties of $\C$.

\benu
\item[(i)] $\C$ is a Krull-Schmidt category.
\item[(ii)] $\C$ is a triangulated category, whose suspension functor over
$\C$ is denoted by $[1].$
\item[(iii)] $\C$ has  Auslander-Reiten triangles, which are  induced by
the Auslander-Reiten triangles of $\D$. We also denoted the
Auslander-Reiten translation of $\C$ by $\tau.$
\enu

\begin{obs}\label{caminoC}
\emph{We deduce by $(iii)$ that the irreducible
morphisms in $\C$ are induced by the irreducible morphisms in
$\D$. Moreover, the non-zero paths of irreducible morphisms between
indecomposable objects in $\C$ are induced by non-zero paths
of irreducible morphisms between
indecomposable objects in $\D$, and both paths have the same length.}
\end{obs}

We denote by $\mathcal{S}$ the set $\mbox{ind}(\mbox{mod}\,H\vee H[1])$
consisting of the indecomposable $H-$modules together with the objects $P[1]$,
where $P$ is an indecomposable projective $H$-module. We can see the set $\mathcal{S}$
as the fundamental domain of $\C$ for the action of $F$ in $\D$, containing
exactly one representative object from each $F$-orbit in $\mbox{ind}\,\D$.

It is known that given $X$ and $Y$ objects in $\mathcal{S}$, then
$\mbox{Hom}_{\D}(F^iX,Y)=0$ for all $i\neq -1,0.$ Moreover, if $H$ is a
representation-finite algebra, then at least one,
$\mbox{Hom}_{\D}(F^{-1}{X},{Y})$ or  $\mbox{Hom}_{\D}({X},{Y})$ vanishes.
\end{sinob}
\vspace{.1in}

\begin{sinob}{\bf On tilting objects}
\vspace{.1in}

An object $T$ in $\C$ is said to be a \textit{tilting object} if $\mbox{Ext}^1_{\C}(T,T)=0$
and $T$ is maximal with that property, that is, if $\mbox{Ext}^1_{\C}(T\oplus X,T\oplus X)=0$
then $X\in \mbox{add}\,T$.

We say that an object $\overline{T}$ in $\C$ is an \textit{almost complete tilting object}
if $\mbox{Ext}_{\C}^1(\overline{T},\overline{T})=0$ and there is an indecomposable
object $X$, which is called \textit{complement} of $\overline{T}$, such that $\overline{T}\oplus X$ is a tilting object.
It is known that an almost complete tilting object $\overline{T}$ in $\C$
has exactly two non-isomorphic complements. We denote them by $M$ and $M^*$.

The algebra $\mbox{End}_{\C}(T)^{op}$, where $T$ is a tilting object in $\C$, is called
a \textit{cluster tilted algebra} of type $\overline{Q}$, where $Q$ is the quiver whose path algebra
is the hereditary algebra $H$, that is, $H=kQ$.

We denote by $\Gamma$ the cluster tilted algebra $\mbox{End}_{\C}(T)^{op}$,
and by $\Gamma'$ the cluster tilted algebra $ \mbox{End}_{\C}(T')^{op}$, with
 $T=\overline{T}\oplus M$ and $T'=\overline{T}\oplus M^*$ where $\overline{T}$ is an almost complete tilting object
in $\C$ with complements $M$ and $M^*$, respectively. In \cite[Theorem 1.3]{BMR1}, the authors proved
that we can pass from one algebra to the other by using mutation.
\vspace{.1in}

The next theorem shall be fundamental  to develop some results of this paper.

\begin{teo}\emph{\cite{BMR}}\label{equivCGamma}
Let $T$ be a tilting object in $\C$ and we denote by $G$ the functor $\emph{Hom}_{\C}(T,-):\C\fle \emmG$.
Then, the functor $\overline{G}:\C/\emph{add}(\tau T)\fle \emmG$ (induced by $G$)
is an equivalence.
\end{teo}

It follows from the above theorem that an indecomposable projective $\Gamma-$module $P_u$
is of the form $\mbox{Hom}_{\C}(T,T_u)$, where $T_u$ is an indecomposable summand of $T$.
Moreover, it is known that the indecomposable injective $\Gamma-$module
$I_u$, which is the injective cover of the simple $S_u=\mbox{top}\,P_u$, is of the
form $\mbox{Hom}_{\C}(T,\tau^2T_u).$

Furthermore, the Auslander-Reiten sequences in $\mbmG\simeq \C/{\mbox{add}(\tau T)}$
are induced by the Auslander-Reiten triangles in $\C$. We can deduce that the irreducible
morphisms in $\mbmG$ which do not factor through $\mbox{add}(\tau T)$ are induced
by irreducible morphisms in $\C.$

Consequently, a path of irreducible morphisms between indecomposable
modules in $\mbmG$ is induced by a path of irreducible morphism between indecomposable objects in $\C$,
and both have the same length.
\vspace{.1in}

Finally, we recall the following important results  useful for our further considerations.

\begin{prop}
Let $\overline{T}$ be a cluster tilted object in $\C$ with complements $M$ and $M^*$.
We consider $\Gamma=\emeCT$ and $\Gamma'=\emph{End}_{\C}(T')^{op}$ with $T=\overline{T}\oplus M$
and $T'=\overline{T}\oplus M^*$.
Then,
\benu
\item[(a)] The $\Gamma$-module $\emph{Hom}_{\C}(T,\tau M^*)$ is simple.
Moreover,  $\emph{Hom}_{\C}(T,\tau M^*)\simeq \emph{top}\,P_x$, where $P_x=\emph{Hom}_{\C}(T,M)$.
\item[(b)] The $\Gamma'$-module $\emph{Hom}_{\C}(T',\tau M)$ is simple.
Moreover, $\emph{Hom}_{\C}(T',\tau M)\simeq \emph{top}\,P'_y$,
where $P'_y=\emph{Hom}_{\C}(T',M^*)$.
\enu
\end{prop}

\begin{teo}\label{equivGG'}
Let $\overline{T}$ be an almost complete tilting object in $\C$ with complements $M$ and $M^*$.
We consider $\Gamma=\emeCT$ and $\Gamma'=\emph{End}_{\C}(T')^{op}$ as above.
Let $S_x$ and $S'_{y}$ be the simples modules $\emph{top}(\emph{Hom}_{\C}(T,M))$
and  $\emph{top}(\emph{Hom}_{\C}(T',M^*))$, respectively. Then, there is an equivalence
$$\theta:\emmG/\emph{add}\,S_x\fle \emmG'/\emph{add}\,S'_{y}.$$
\end{teo}

\begin{obs}\label{obsGG'}
\emph{We consider $\Gamma$ and $\Gamma'$ to be the cluster tilted algebras
as in the  above theorem. Let $T_a$ be a direct summand of $\overline{T}$
and we denote by $P_a=\mbox{Hom}_{\C}(T,T_a)$ and $I_a=\mbox{Hom}_{\C}(T,\tau^2T_a)$
the indecomposable projective and injective $\Gamma$-modules corresponding to the vertex $a\in Q_{\Gamma}$,
respectively. We denote by $P'_a=\mbox{Hom}_{\C}(T',T_a)$ and $I'_a=\mbox{Hom}_{\C}(T',\tau^2T_a)$
the indecomposable projective and injective $\Gamma'$-modules corresponding to the vertex $a\in Q_{\Gamma'}$,
respectively.}

\emph{Following the equivalence of Theorem \ref{equivGG'}, note that
it is not hard to see that $\theta(P_a)=P'_a$ and $\theta(I_a)=I'_a$.}
\end{obs}

\begin{nota}\label{numu}\emph{
For a better understanding of the results, when we consider the cluster tilted algebras
$\Gamma\simeq kQ_{\Gamma}/I_{\Gamma}$  and
$\Gamma'\simeq kQ_{\Gamma'}/I_{\Gamma'}$, in order to compute the nilpotency indices of
$\Re(\mbmG)$ and $\Re(\mbmG')$, we denote such values by
 $r_u=m_u+n_u$, for each $u\in Q_{\Gamma}$,
and by  $r'_v=m'_v+n'_v$ for each $v\in Q_{\Gamma'}$, as  defined in \ref{nilpo}.}
\end{nota}
\end{sinob}

\section{The main results}

In this section, we compute the nilpotency index of the radical of the
module category of a representation-finite cluster tilted algebra.

Let  $H$ be a hereditary algebra and $\Gamma=\mbeCT$
the cluster tilted algebra, where $T$ is a tilting object in the cluster category
$\C=\D/F$, $\D$ is the bounded derived category of $\mbmH$ and $F=\tau^{-1}_{\D}[1]$.

It is well known that $\Gamma$ is representation-finite if and only if $H$ so is.
In this case, $H\simeq k\Delta$ with $\overline{\Delta}$ a Dynkin graph and the
Auslander-Reiten quiver, $\Gamma(\D)$, of $\D$ is isomorphic to $\mathbb{Z}\Delta$.

\begin{prop}
Let $H$ be a representation-finite hereditary algebra. Then the Auslander-Reiten quiver
$\Gamma (\D^b(\emph{mod}\,H))$ is a component with length.
\end{prop}

\begin{proof}
We analyze the orbit graph of  $\Gamma (\D^b(\mbox{mod}\,H))$.
Since $H$ is a representation-finite hereditary algebra, we have that
$H\simeq k\Delta$, with $\overline{\Delta}$ a Dynkin graph.
Moreover, $\Gamma (\D^b(\mbox{mod}\,H))\simeq \mathbb{Z}\Delta$.
It is clear that the orbit graph of $\Gamma (\D^b(\mbox{mod}\,H))$
is isomorphic to $\overline{\Delta},$ which is of tree type.
Therefore, $\Gamma (\D^b(\mbox{mod}\,H))$ is a simply connected
translation quiver and therefore  $\Gamma (\D^b(\mbox{mod}\,H))$ is a component with length.
\end{proof}

In the next result, we give a relationship between morphisms of the categories $\mbmG$ and $\mbmG'$.

\begin{prop}\label{irredcluster}
Let $\C$ be the cluster category of a hereditary algebra $H$ and
let $\overline{T}$ be an almost complete tilting object in $\C$ with complements $M$ and $M^*$.
Consider $\Gamma=\emeCT$ and
$\Gamma'=\emph{End}_{\C}(T')^{op}$ with $T=\overline{T}\oplus M$
and $T'=\overline{T}\oplus M^*$. Let  $S_x$ and  $S'_{y}$ be
the simple tops of $\emph{Hom}_{\C}(T,M)$ and $\emph{Hom}_{\C}(T',M^*)$,
respectively, and let $\theta:\emmG/\emph{add}\,S_x\fle \emmG'/\emph{add}\,S'_{y}$
be the equivalence of Theorem  \ref{equivGG'}.

Consider  a morphism  $f:X\fle Y$  with $X,Y\in \emiG$. Then,  $f$ is an irreducible morphism in $\emmG$
which does not factor through $\emph{add}\,S_{x}$
if and only if $\theta(f)$ is an irreducible morphism in  $\emmG'$ which does not factor through $\emph{add}\,S'_{y}.$
\end{prop}

\begin{proof}
Let  $\Gamma$ and $\Gamma'$ be the cluster tilted algebras as above. Let $X$ and $Y$ be
indecomposable modules in $\mbmG$ and let $f:X\fle Y$ be a non-zero morphism such that
$f$ does not factor through $\mbox{add}\,S_{x}$. By the equivalence of the Theorem \ref{equivGG'},
we have that $\theta(f):\theta(X)\fle \theta(Y)$ is a non-zero morphism and moreover $\theta(f)$ does not factor through $\emph{add}\,S'_{y}$

Assume that $f$ is irreducible. We prove that $\theta(f)$ so is.
In fact, assume that $\theta(f)$ is a section. Then there exists
a morphism  $\til{f'}:\theta(Y)\fle \theta(X)$ such that $\til{f'}\theta(f)=1_{\theta(X)}$.
Moreover, $\til{f'}$ does not factor through $\mbox{add}\,S'_{y}$
because $\theta(f)$ neither does.
Then, there is a morphism $f':Y\fle X$ such that $\til{f'}=\theta(f')$.
Therefore, $$\theta(1_X)=1_{\theta(X)}=\theta(f')\theta(f)=\theta(f'f)$$
\noindent and since $\theta$ is a faithful functor, then
$1_X=f'f$, which is a contradiction to the fact that $f$ is not a section.
Thus, we prove that $\theta(f)$ is not a section.

Analogously, we can prove that $\theta(f)$ is not a retraction.

Now, assume that the is a $\Gamma'$-module $\til{Z}$ and that there are morphisms
$\til{g}:\theta(X)\fle \til{Z}$ and $\til{h}:\til{Z}\fle \theta(Y)$, such that
$\theta(f)=\til{h}\til{g}$. Since $\theta(f)$ does not factor though $\mbox{add}\,S'_{y}$, we infer that
neither do the morphisms $\til{g}$ y $\til{h}$.
By Theorem \ref{equivGG'}, there exist $Z\in \mbmG$ and morphisms
$g:X\fle Z$ and $h:Z\fle Y$ which do not factor trough $\mbox{add}\,S_{x}$ such that
$\til{g}=\theta(g)$ and $\til{h}=\theta(h)$. Then, $\theta(f)=\theta(h)\theta(g)=\theta(hg)$
and consequently $f=hg$. Since $f$ is an irreducible morphism,
then $g$ is a section (and therefore $\theta(g)$ also does) or
$h$ is a retraction (and therefore $\theta(h)$ also does). Thus,
$\theta(f)$ is an irreducible morphism.

The converse follows by considering  $\theta'$ the quasi-inverse equivalence of $\theta$.
\end{proof}

Our next goal is to prove that the nilpotency index
is invariant under mutation.
\vspace{.1in}

Following the above notation, we donote by $a$  the vertex of $Q_{\Gamma}$ and  of $Q_{\Gamma'}$,
which come from $T_a$, a direct indecomposable summand of  $\overline{T}$,
and we denote by $x$ ($y$, respectively) the vertex of $Q_{\Gamma}$  ($Q_{\Gamma'}$, respectively)
which come from $M$, the summand of $T$ ($M^*$, the summand of $T'$, respectively).
\vspace{.1in}

We start with some lemmas in order to prove one of the main theorems of this section.

\begin{lem}\label{ra=r'a}
Let $\C$ be a cluster category of a representation-finite hereditary algebra $H$, and
let $\overline{T}$ be an almost complete tilting object in $\C$ with complements $M$ and $M^*$.
Consider $\Gamma=\emph{End}_{\C}(\overline{T}\oplus M)^{op}\simeq kQ_{\Gamma}/I_{\Gamma}$ and
$\Gamma'=\emph{End}_{\C}(\overline{T}\oplus M^*)^{op}\simeq kQ_{\Gamma'}/I_{\Gamma'}$
the cluster tilted algebras.
Then, for all indecomposable summand  $T_a$ of $\overline{T}$, we have that $r_a=r'_a$.
\end{lem}

\begin{proof}
Let $\Gamma$ and $\Gamma'$ be cluster tilted algebras as in the statement, and
let $T_a$ be an indecomposable summand of $\overline{T}$. Consider
$P_a$, $S_a$  and $I_a$ the projective, simple and injective $\Gamma$-modules, respectively,
corresponding to the vertex $a\in Q_{\Gamma}$, and $P'_a$, $S'_a$  and  $I'_a$
the projective, simple and injective $\Gamma'$-modules, respectively,
corresponding to the vertex $a\in Q_{\Gamma'}$.
Let $r_a$ and $r'_a$ be the bounds defined in Notation \ref{numu}.
We prove that $r_a=r'_a.$

Let $f_a:P_a\fle I_a$ be a non-zero morphism in $\mbmG$ that
factors trough $S_a$. By Lemma \ref{clau2.5},  we have that
${\rm dp}(f_a)=r_a.$
Therefore, by Proposition \ref{aus}, we can write the morphism $f_a$ as follows
$$f_a=\Sigma_{i=1}^sg_if_i$$
\noindent  for some $s\geq 1$, where $f_i\in \Re_{\Gamma}(P_a,X_i)$, with $X_i\in \mbiG$,
and $g_i$ is a finite sum of composition of $r_a-1$ irreducible morphisms between indecomposable modules, for $i=1, \dots, s$.

Let $S_x$ be the simple top of the projective  $\Gamma$-module $P_x=\mbox{Hom}_{\C}(T,M)$.
Since $S_a\neq S_x$, neither $f_i$ nor $g_i$ factor through
$\mbox{add}\,S_x$, because $\mbox{Hom}_{\Gamma}(P_a,S_x)=0=\mbox{Hom}_{\Gamma}(S_x,I_a)$.
Then, by the equivalence $\theta:\mbmG/\mbox{add}\,S_x\fle \mbmG'/\mbox{add}\,S_{y}$ defined above,
we have that $\theta(f_a)=\Sigma_{i=1}^s\theta(g_i)\theta(f_i)$ is a non-zero morphisms, where each
$\theta(f_i)\in \Re_{\Gamma'}(\theta(P_a),\theta(X_i))$. Moreover, by Proposition \ref{irredcluster},
each $\theta(g_i)$ is a finite sum of composition of $r_a-1$ irreducible morphisms between indecomposable modules.
Then, $\theta(f_a)\in \Re_{\Gamma'}^{r_a}(\theta(P_a),\theta(I_a))$, that is,
 $\theta(f_a)\in \Re_{\Gamma'}^{r_a}(P'_a,I'_a)$. By Lemma \ref{clau2.5} we have that  $r_a\leq r'_a$.

Similarly, we can prove that $r'_a\leq r_a$. Hence,  $r_a=r'_a$ as we wish to prove.
\end{proof}

\begin{lem}\label{rx=r'y}
Let $\C$ be the cluster category of a representation-finite hereditary algebra $H$.
let $\overline{T}$ be an almost complete tilting object in $\C$ with complements $M$ and $M^*$.
Consider $\Gamma=\emph{End}_{\C}(\overline{T}\oplus M)^{op}\simeq kQ_{\Gamma}/I_{\Gamma}$ and
$\Gamma'=\emph{End}_{\C}(\overline{T}\oplus M^*)^{op}\simeq kQ_{\Gamma'}/I_{\Gamma'}$
the cluster tilted algebras.
Let $x$ and $y$ be the vertices of $Q_{\Gamma}$  and $Q_{\Gamma'}$, respectively
which come from the summands $M$ of $T$ and $M^*$ of $T'$, respectively.
Then $r_x=r'_y$.
\end{lem}

\begin{proof}
Let $\Gamma\simeq kQ_{\Gamma}/I_{\Gamma}$ and $\Gamma'\simeq kQ_{\Gamma'}/I_{\Gamma'}$
be the cluster tilted algebras defined as above.

To prove that $r_x=r'_y$, we shall prove the fact that  $n_x=m'_{y}$ and $m_x=n'_{y},$ where $r_x=n_x+m_x$ and
$r'_y=n'_y+m'_y$ are the bounds  defined in Notation \ref{numu}.

Consider a non-zero morphism $f_x:P_x\fle S_x$ and the irreducible morphism $i_x:\mbox{rad}P_x \rightarrow P_x$. Since the cluster tilted algebra $\Gamma$ is representation-finite then by \cite[Theorem A]{CLT} we know that the right degree of  $i_x$ is finite and more precisely it is $n_x$. Therefore, by the dual result of \cite[Proposition 3.4]{CLT},
we have that ${\rm dp}(f_x)=n_x.$
Hence, by Proposition \ref{aus} we know that there is a non-zero path
of irreducible morphisms between indecomposable modules of length $n_x$ in $\mbmG$ as follows

$$\phii_x:P_x\stackrel{h_1} \longrightarrow X_1 \stackrel{h_2} \longrightarrow X_2\longrightarrow \dots \longrightarrow X_{n_x-1}\stackrel{h_{n_x}} \longrightarrow S_x.$$

By the equivalence defined in Theorem \ref{equivCGamma},
it is induced by a non-zero path of also $n_x$ irreducible morphisms between indecomposable objects in the cluster
category $\C$, such that it does not factor trough $\mbox{add}\,\tau T$

\begin{equation}\label{camphii}
\til{\phii}_x:{M}\stackrel{\til{h}_1} \longrightarrow \til{X}_1 \stackrel{\til{h}_2} \longrightarrow \til{X}_2\longrightarrow \dots \longrightarrow \til{X}_{n_x-1}\stackrel{\til{h}_{n_x}} \longrightarrow {\tau M^{*}}
\end{equation}

\noindent where  $P_x=\mbox{Hom}_{\C}(T,{M})$,
$S_x=\mbox{Hom}_{\C}(T,\tau M^{*})$ and  $X_i=\mbox{Hom}_{\C}(T,\til{X}_i)$ for
$1\leq i\leq n_x-1$.

On the other hand, if we consider a non-zero morphism $g'_{y}:S'_{y}\fle I'_{y}$ in $\mbmG'$,
by Theorem A and  Proposition 3.4 in \cite{CLT},  we have that ${\rm dp}(g'_y)=m'_y.$
Hence, with an analogous analysis to the previous one, there exists a non-zero path $\psi'_{y}$
of $m'_y$ irreducible morphisms between indecomposable modules from $S'_{y}$ to $I'_{y}$ in $\mbmG'$.
Moreover, such a path is induced by a non-zero path $\til{\psi'}_{y}$, from $\tau M$ to $\tau^2M^*$, of
$m'_y$ irreducible morphisms between indecomposable modules in the cluster category $\C$ and such that it does not
factor trough $\mbox{add}\,T'$

\begin{equation}\label{campsi}
\til{\psi'}_{y}:\tau M\fle \til{Y'}_1\fle \til{Y'}_2\fle \dots \fle \til{Y'}_{m_{y}-1}\fle \tau^2M^*
\end{equation}

\noindent because
$S'_{y}=\mbox{Hom}_{\C}(T',\tau{M})$ and $I'_{y}=\mbox{Hom}_{\C}(T',\tau^2{M^*})$.

We also have that
$\til{\phii}_x\in \mbox{Hom}_{\C}(M,\tau M^*)$, with $ \til{\phii}_x\neq 0$, where

$$\mbox{Hom}_{\C}(M,\tau M^*)=\mbox{Hom}_{\D}(F^{-1}M,\tau M^*)\oplus \mbox{Hom}_{\D}(M,\tau M^*);$$

\noindent and $\til{\psi'}_{y}\in \mbox{Hom}_{\C}(\tau M,\tau^2 M^*)$, with $\til{\psi'}_{y}\neq 0$, where

$$\mbox{Hom}_{\C}(\tau M,\tau^2 M^*)=\mbox{Hom}_{\D}(F^{-1}\tau M,\tau^2 M^*)\oplus \mbox{Hom}_{\D}(\tau M,\tau^2 M^*).$$

\noindent In both cases, only one of the summands is non-zero since $H$ is
representation-finite.

Hence, if $\mbox{Hom}_{\D}(F^{-1}M,\tau M^*)\neq 0$, then
$\mbox{Hom}_{\D}(F^{-1}\tau M,\tau^2 M^*)\neq 0$ because

\begin{displaymath}
\begin{array}{ccl}
\mbox{Hom}_{\D}(F^{-1}M,\tau M^*) & = &\mbox{Hom}_{\D}(\tau M[{-1}],\tau M^*)\\
 & \simeq &\mbox{Hom}_{\D}(\tau^2 M[-1],\tau^2 M^*)\\
 & \simeq &\mbox{Hom}_{\D}(F^{-1}\tau M,\tau^2 M^*).
\end{array}
\end{displaymath}

Therefore, the path in (\ref{camphii}) is induced by a path of irreducible morphisms between indecomposable modules of length $n_x$
from $\tau M[{-1}]$ to $ \tau M^*$  in $\D$,
and the path in  (\ref{campsi}) is induced by a path of $m'_y$ irreducible morphisms
between indecomposable modules from $\tau^2 M[{-1}]$ to $\tau^2 M^*$ in $\D$.
Moreover, since $\mbox{Hom}_{\D}(\tau M[{-1}],\tau M^*)\simeq \mbox{Hom}_{\D}(\tau^2 M[-1],\tau^2 M^*)$ and
$\Gamma(\D)$ is a translation quiver with length,  then $n_x=m'_{y}$.

Now, if $\mbox{Hom}_{\D}(M,\tau M^*)\neq 0$, with the same argument as before we can conclude that $n_x=m'_{y}$.

Analogously, considering a non-zero morphism $g_x:S_x\fle I_x$
in $\mbmG$ and a non-zero morphism $f'_{y}:P'_{y}\fle S'_{y}$
in $\mbmG'$, with a similar analysis as above, we conclude that  $m_x=n'_{y}.$
Thus, $r_x=m_x+n_x=n'_{y}+m'_{y}=r'_{y}.$
\end{proof}

Now, we are in position to show one Theorem A.

\begin{teo}\label{nilcluster}
Let $\C$ be the cluster category of a representation-finite hereditary algebra $H$.
Consider $\Gamma=\emph{End}_{\C}(\overline{T}\oplus M)^{op}$ and
$\Gamma'=\emph{End}_{\C}(\overline{T}\oplus M^*)^{op}$ the cluster
tilted algebras, where $\overline{T}$ is an almost complete tilting object in $\C$ with complements $M$ and $M^*$.
 Let $r_{\Gamma}$ and $r_{\Gamma'}$
be the nilpotency indices of $\Re(\emmG)$ and $\Re(\emmG')$, respectively.
Then, $r_{\Gamma}=r_{\Gamma'}$.
\end{teo}

\begin{proof}
Let  $\Gamma\simeq kQ_{\Gamma}/I_{\Gamma}$ and $\Gamma'\simeq kQ_{\Gamma'}/I_{\Gamma'}$ be the
cluster tilted algebras defined as above.
Since $H$ is a representation-finite algebra, then so are $\Gamma$ and $\Gamma'$.
We denote by $r_{\Gamma}$ and  $r_{\Gamma'}$, the nilpotency indices of
$\Re(\mbmG)$ and $\Re(\mbmG'),$ respectively. We  prove that $r_{\Gamma}=r_{\Gamma'}$.
In fact, we know that

$$r_{\Gamma}=\mbox{max}\{r_u\mid u\in (Q_{\Gamma})_0\}+1=\mbox{max}\{r_u\mid T_u\in \mbox{ind}(\mbox{add}\,T)\}+1$$

\noindent and

$$r_{\Gamma'}=\mbox{max}\{r'_v\mid v\in (Q_{\Gamma'})_0\}+1=\mbox{max}\{r'_v\mid T_v\in \mbox{ind}(\mbox{add}\,T')\}+1.$$

By Lemma \ref{ra=r'a} and Lemma \ref{rx=r'y}, we have that

\begin{displaymath}
\begin{array}{ccl}
r_{\Gamma} & = &\mbox{max}\{r_u\mid T_u\in \mbox{ind}(\mbox{add}\,T)\}+1\\
 & = &\mbox{max}\{r_a\mid T_a\in \mbox{ind}(\mbox{add}\,\overline{T}),r_x\}+1\\
 & = &\mbox{max}\{r'_a\mid T_a\in \mbox{ind}(\mbox{add}\,\overline{T}),r'_y\}+1\\
& = &\mbox{max}\{r_v\mid T_v\in \mbox{ind}(\mbox{add}\,T')\}+1\\
&= &r_{\Gamma'},
\end{array}
\end{displaymath}

\noindent proving the result.
\end{proof}

For the convenience if the reader we state \cite[Theorem 4.11]{Z}.

\begin{teo}\label{ppalher}
Let $H=k\Delta$ be a representation-finite hereditary algebra and let $r_H$
be the nilpotency index of $\Re(\emph{mod}\,H).$ Then the  following conditions hold.
\begin{enumerate}
\item[(a)] If $\overline{\Delta}=A_n$, then $r_H=n,$ for $n\geq 1$.
\item[(b)] If $\overline{\Delta}=D_n$, then $r_H=2n-3,$ for $n\geq 4$.
\item[(c)] If $\overline{\Delta}=E_6$, then $r_H=11.$
\item[(d)] If $\overline{\Delta}=E_7$, then $r_H=17.$
\item[(e)] If $\overline{\Delta}=E_8$, then $r_H=29.$
\end{enumerate}
\end{teo}

The next corollary shall be important to prove Theorem B, and follows from \cite{BMR} and \cite{BMR1}.

\begin{coro}\label{mutDynkin}
Let $\Delta$ be a connected and  acyclic quiver. The classes of quivers obtained of $\Delta$ by mutations
coincide with the classes of quivers of the cluster tilted algebras of type
$\Delta$. Moreover, if $\Delta$ is of Dynkin type then there is a finite number of the mentioned classes.
\end{coro}

Now, we are in conditions to present Theorem B.

\begin{teo}\label{clusterppal}
Let ${\Delta}$ be a Dynkin quiver and let $\Gamma$ be a cluster tilted algebra
of type $\overline{\Delta}$. Let $r_{\Gamma}$ be  the nilpotency index of $\Re(\emmG)$. Then the  following conditions hold.
\begin{enumerate}
\item[(a)] If $\overline{\Delta}=A_n$, then $r_{\Gamma}=n$ for $n\geq 1$.
\item[(b)] If $\overline{\Delta}=D_n$, then $r_{\Gamma}=2n-3$ for $n\geq 4$.
\item[(c)] If $\overline{\Delta}=E_6$, then $r_{\Gamma}=11$.
\item[(d)] If $\overline{\Delta}=E_7$, then $r_{\Gamma}=17.$
\item[(e)] If $\overline{\Delta}=E_8$, then $r_{\Gamma}=29.$
\end{enumerate}
\end{teo}

\begin{proof}
Let $\Gamma\simeq kQ_{\Gamma}/I_{\Gamma}$ be a cluster tilted algebra of type $\overline{\Delta}$,
where ${\Delta}$ is a Dynkin quiver and let $H$ be the hereditary algebra $H=k\Delta$.
Since $H$ is representation-finite, then so is $\Gamma$.

Let $r_{H}$ and $r_{\Gamma}$ the nilpotency indices of $\Re(\mbmH)$ and $\Re(\mbmG)$,
respectively. We claim that $r_{\Gamma}=r_{H}$.
In fact, by Corollary \ref{mutDynkin}, we can change the algebra $\Gamma$ into
the algebra $H$ by a finite sequence of mutations of the quiver $Q_{\Gamma}$.
By Theorem \ref{nilcluster}, we have that $r_{\Gamma}=r_{H}$ where $r_{H}$ takes the
values given in Theorem \ref{ppalher}. Hence we prove the statement.
\end{proof}

In \cite{R}, C. Ringel proved that all self-injective cluster tilted algebras
are representation-finite. Furthermore, the author showed that this particular algebras are cluster tilted algebras of type $D_n$, with $n\geq 3$ (considering $D_3=A_3$).

The next result follows immediately from Theorem \ref{clusterppal}.

\begin{coro}
Let $\Gamma$ be a self-injective cluster tilted
algebra. Then, the nilpotency index of $\Re(\emmG)$ is $2n-3$, where
$n$ is the number of vertices of the quiver $Q_{\Gamma}$.
\end{coro}

We end up this section with a remark on Coxeter numbers. We refer the reader to \cite{FR} for a detailed account on
Root Systems and Coxeter Groups.

\begin{obs}
\emph{It is known that the theory of cluster algebras has many connections with different
areas in mathematics. In particular, there exists a connection with Root Systems
and with Coxeter Groups.}

\emph{An element in a Coxeter group $W$ is called a coxeter element if it is the product of
all simple reflections and moreover its order is called the
coxeter number of $W$. On the other hand, the coxeter number
is related with the highest root in its corresponding root system.}

\emph{For a finite irreducible Coxeter group $W$, there is a corresponding
root system $\Phi$ of  Dynkin type $\Delta$. Now, if  $\Gamma$ is a cluster
tilted algebra of $\Delta$ type then it is known that the cardinal of the set of positive
roots of $\Phi$ coincide with the cardinal of $\mbox{ind}\,\Gamma$.  Moreover,
the coxeter number of $W$ is exactly one more than the nilpotency index of
$\Re(\mbox{mod}\,\Gamma)$.}
\end{obs}

\section{On composition of irreducible morphisms}

In this section, we establish the relationship  between the composition of irreducible
morphisms between indecomposable modules and the power of the radical
where it belongs.

We start with  the following proposition.

\begin{prop}\label{dimhomc}
Let $\Gamma$ be a representation-finite cluster tilted algebra. Let $M$
and $N$ be indecomposable $\Gamma$-modules such that $\emph{Irr}_{\Gamma}(M,N)\neq 0$.
Then, $\emph{dim}_k(\emph{Hom}_{\Gamma}(M,N))=1$. In particular, $\Re_{\Gamma}^2(M,N)=0$.
\end{prop}

\begin{proof}
Let $\Gamma$ be a representation-finite cluster tilted algebra.
Then $\Gamma=\mbeCT$, where $\C=\D/F$ is the cluster category of a representation-finite
hereditary algebra $H$ and $T$ a tilting object in $\C$.

Since $\mbox{Irr}_{\Gamma}(M,N)\neq 0$, then there exists an irreducible morphism $f:M\fle N$.
We claim that all the morphisms $g:M\fle N$ in $\mbmG$ are $k$-linearly dependent with $f$.
In fact, suppose that there exists a non-zero morphisms $g:M\fle N$ $k$-linearly independent with $f$.
Since $\Gamma$ is representation-finite, we know that  $\mbox{dim}_k(\mbox{Irr}_{\Gamma}(M,N))= 1.$
Hence, $g$ is not irreducible. Then $g\in \Re^2(M,N)$. Moreover, there is a integer $n\geq 2$ such that $g\in  \Re^n(M,N)\backslash \Re^{n+1}(M,N).$
Therefore, there exist morphisms $\til{f},\,\til{g}:\til{M}\fle \til{N}$ in the cluster
category $\C$ such that they do not factor through $\mbox{add}(\tau\,T)$. Furthermore,
these morphisms are induced by morphisms in the derived category. Moreover,
since $\mbox{Hom}_{\C}(\til{M},\til{N})=\mbox{Hom}_{\D}(F^{-1}M,N)\oplus \mbox{Hom}_{\D}(M,N),$ and
only one of this summands are non-zero, we can deduce the existence of an irreducible morphism in  $\Gamma(\D)$
and a path of length $n$, with $n\geq 2$, between the same objects, contradicting the fact that $\Gamma(\D)$
is a quiver with length.

Therefore, there is not a morphism  $g:M\fle N$  in $\mbmG$ linearly independent with $f$.
Then, $\mbox{dim}_k(\mbox{Hom}_{\Gamma}(M,N))=1$. Moreover, since $f$ is irreducible,
we conclude that $\Re_{\Gamma}^2(M,N)=0$.
\end{proof}

\begin{teo}
Let $\Gamma$ be a representation-finite cluster tilted algebra. Consider the  irreducible morphisms
$h_i:X_i\fle X_{i+1}$, with $X_i\in \emiG$ for $1\leq i\leq m$.
Then $h_{m}\dots h_{1}\in \Re^{m+1}(X_{1}, X_{m+1})$ if and only if $ h_{m}\dots h_{1} = 0$.
\end{teo}

\begin{proof}
If $ h_{m}\dots h_{1} = 0$, then clearly we have that $h_{m}\dots h_{1}\in \Re^{m+1}(X_{1}, X_{m+1})$.

Conversely. Assume that $h_{m}\dots h_{1}\in \Re^{m+1}(X_{1}, X_{m+1})$ and $ h_{m}\dots h_{1} \neq 0.$
Then, by \cite[Theorem 5.1]{CLT} there are irreducible morphisms $f_i:X_i\fle X_{i+1}$,
for $1\leq i\leq m$ such that $ f_{m}\dots f_{1} = 0.$ By Proposition \ref{dimhomc}, we have that
$\mbox{dim}_k(\mbox{Hom}_{A}(X_i,X_{i+1}))=1,$ for each $i$. Hence, $f_i$ and $h_i$
are $k$-linearly dependent, that is, $h_i=\lambda_if_i$ where $\lambda_i$ is a non-zero element of $k$.
Thus, $ h_{m}\dots h_{1}= \lambda f_{m}\dots f_{1}=0,$ which is a contradiction to our assumption.
Therefore, $ h_{m}\dots h_{1} = 0$ proving the result.
\end{proof}

\begin{obs}
We observe that if  we consider a cluster tilted algebra of type $A_n$ or $D_n$, then
the results of this article can be proven  with the geometric approach developed for cluster categories and cluster tilted algebras of type $A_n$ and $D_n$ given
in \cite{CCS} and in \cite{S}, respectively. For a detail account on this approach see \cite{VG}.
\end{obs}

\end{document}